\documentclass[a4paper, 12pt]{article}

\usepackage{amsmath}
\usepackage{amssymb}
\usepackage{amsthm}

\begin{document}

\title{$\ast$-transforms of acyclic complexes}
\author{Taro Inagawa}
\date{
\small
Graduate School of Science, Chiba University, \\
1-33 Yayoi-Cho, Inage-Ku, Chiba-Shi, 263-8522, JAPAN
}
\maketitle

\newcommand{\lra}{\longrightarrow}

\newcommand{\lsa}[1]{{}^{\ast}\!{#1}}
\newcommand{\lsp}[1]{{}'\!{#1}}
\newcommand{\lspp}[1]{{}''\!{#1}}

\newcommand{\brac}[1]{[{#1}]}
\newcommand{\ang}[1]{\langle{#1}\rangle}

\newtheorem{thm}{Theorem}[section]
\newtheorem{prop}[thm]{Proposition}
\newtheorem{cor}[thm]{Corollary}
\newtheorem{dfn}[thm]{Definition}
\newtheorem{lem}[thm]{Lemma}
\newtheorem{rem}[thm]{Remark}
\newtheorem{ex}[thm]{Example}
\newtheorem{claim}{Claim}

\begin{abstract}
Let $R$ be an $n$-dimensional Cohen-Macaulay local ring and $Q$ a parameter ideal of $R$.
Suppose that an acyclic complex $(F_{\bullet}, \varphi_{\bullet})$ of length $n$ of finitely generated free $R$-modules is given.
We put $M = {\rm Im} \, \varphi_{1}$, which is an $R$-submodule of $F_{0}$.
Then $F_{\bullet}$ is an $R$-free resolution of $F_{0}/M$.
In this paper, we describe a concrete procedure to get an acyclic complex $\lsa{F_{\bullet}}$ of length $n$ that resolves $F_{0}/(M :_{F_{0}} Q)$.
\end{abstract}

\section{Introduction}

The $\ast$-transform of an acyclic complex of length $3$ is introduced in \cite{fin1}.
The purpose of this paper is to generalize it for acyclic complexes of length $n \geq 2$.
Let $(R, \mathfrak{m})$ be an $n$-dimensional Cohen-Macaulay local ring and $x_{1}, x_{2}, \dots, x_{n}$ an sop for $R$.
We put $Q = (x_{1}, x_{2}, \dots, x_{n})R$.
Suppose that an acyclic complex
\[
0 \lra F_{n} \stackrel{\varphi_{n}}{\lra} F_{n-1} \lra \dotsb \lra F_{1} \stackrel{\varphi_{1}}{\lra} F_{0}
\]
of finitely generated free $R$-modules such that $\text{Im} \, \varphi_{n} \subseteq QF_{n-1}$ is given.
In this paper, we describe an operation to get an acyclic complex
\[
0 \lra \lsa{F_{n}} \stackrel{\lsa{\varphi_{n}}}{\lra} \lsa{F_{n-1}} \lra \dotsb \lra \lsa{F_{1}} \stackrel{\lsa{\varphi_{1}}}{\lra} \lsa{F_{0}} = F_{0}
\]
such that $\text{Im} \, \lsa{\varphi_{1}} = \text{Im} \, \varphi_{1} :_{F_{0}} Q$ and $\text{Im} \, \lsa{\varphi_{n}} \subseteq \mathfrak{m} \cdot \lsa{F_{n-1}}$,
which is called the $\ast$-transform of $F_{\bullet}$ with respect to $x_{1}, x_{2}, \dots, x_{n}$.
As we give a practical condition for $\lsa{F_{n}}$ to be vanished, it is possible to consider when the depth of $F_{0} / (\text{Im} \, \varphi_{1} :_{F_{0}} Q)$ is positive.
This is useful when we compute the saturation of ideals.
In fact, in the subsequent paper \cite{fin2}, using $\ast$-transform we compute the saturation of the $m$-th power of the ideal generated by the maximal minors of the following $m \times (m+1)$ matrix
\[
\begin{pmatrix}
x_{1}^{\alpha_{1,1}}  & x_{2}^{\alpha_{1,2}}   & x_{3}^{\alpha_{1,3}}  & \dotsb & x_{m}^{\alpha_{1,m}}   & x_{m+1}^{\alpha_{1,m+1}} \\
x_{2}^{\alpha_{2,1}}  & x_{3}^{\alpha_{2,2}}   & x_{4}^{\alpha_{2,3}}  & \dotsb & x_{m+1}^{\alpha_{2,m}} & x_{1}^{\alpha_{2,m+1}}   \\
x_{3}^{\alpha_{3,1}}  & x_{4}^{\alpha_{3,2}}   & x_{5}^{\alpha_{3,3}}  & \dotsb & x_{1}^{\alpha_{3,m}}   & x_{2}^{\alpha_{3,m+1}}   \\
\vdots & \vdots  & \vdots &        & \vdots  & \vdots  \\
x_{m}^{\alpha_{m,1}}  & x_{m+1}^{\alpha_{m,2}} & x_{1}^{\alpha_{m,3}}  & \dotsb & x_{m-2}^{\alpha_{m,m}} & x_{m-1}^{\alpha_{m,m+1}}
\end{pmatrix},
\]
where $x_{1}, x_{2}, x_{3}, \dots, x_{m}, x_{m+1}$ is an sop and $\{ \alpha_{i,j} \}_{1 \leq i \leq m, 1 \leq j \leq m+1}$ is a family of positive integers.

Throughout this paper, $R$ is a commutative ring, and in the last section, we assume that $R$ is an $n$-dimensional Cohen-Macaulay local ring.
For $R$-modules $G$ and $H$, the elements of $G \oplus H$ are denoted by column vectors;
\[
\begin{pmatrix} g \\ h \end{pmatrix} \quad \text{($g \in G$, $h \in H$)}.
\]
In particular, the elements of the forms
\[
\begin{pmatrix} g \\ 0 \end{pmatrix} ~ \text{and} ~ \begin{pmatrix} 0 \\ h \end{pmatrix}
\]
are denoted by $\brac{g}$ and $\ang{h}$, respectively.
Moreover, if $V$ is a subset of $G$, then the family $\{ \brac{v} \}_{v \in V}$ is denoted by $\brac{V}$.
Similarly $\ang{W}$ is defined for a subset $W$ of $H$.
If $T$ is a subset of an $R$-module, we denote by $R \cdot T$ the $R$-submodule generated by $T$.
If $S$ is a finite set, $\sharp \, S$ denotes the number of elements of $S$.

\section{Preliminaries}

In this section, we summarize preliminary results. Let $R$ be a commutative ring.

\begin{lem}\label{1}
Let $G_{\bullet}$ and $F_{\bullet}$ be acyclic complexes, whose boundary maps are denoted by $\partial_{\bullet}$ and $\varphi_{\bullet}$, respectively.
Suppose that a chain map $\sigma_{\bullet} : G_{\bullet} \lra F_{\bullet}$ is given and $\sigma_{0}^{-1}({\rm Im} \, \varphi_{1}) = {\rm Im} \, \partial_{1}$ holds.
Then the mapping cone ${\rm Cone}(\sigma_{\bullet}) :$
\[
\dotsb
\lra
\begin{matrix} G_{p-1} \\ \oplus \\ F_{p} \end{matrix}
\stackrel{\psi_{p}}{\lra}
\begin{matrix} G_{p-2} \\ \oplus \\ F_{p-1} \end{matrix}
\lra
\dotsb
\lra
\begin{matrix} G_{1} \\ \oplus \\ F_{2} \end{matrix}
\stackrel{\psi_{2}}{\lra}
\begin{matrix} G_{0} \\ \oplus \\ F_{1} \end{matrix}
\stackrel{\psi_{1}}{\lra}
F_{0}
\lra
0
\]
is acyclic, where
\[
\psi_{p} = \begin{pmatrix} \partial_{p-1} & 0 \\ (-1)^{p-1} \cdot \sigma_{p-1} & \varphi_{p} \end{pmatrix}
~ {\it for} ~ \forall p \geq 2 ~ {\it and} ~
\psi_{1} = \begin{pmatrix} \sigma_{0} & \varphi_{1} \end{pmatrix}.
\]
Hence, if $G_{\bullet}$ and $F_{\bullet}$ are complexes of finitely generated free $R$-modules, then ${\rm Cone}(\sigma_{\bullet})$ gives an $R$-free resolution of ${\rm Im} \, \varphi_{1} + {\rm Im} \, \sigma_{0}$.
\end{lem}

\begin{proof}
See \cite[2.1]{fin1}.
\end{proof}

\begin{lem}\label{2}
Let $2 \leq n \in \mathbb{Z}$ and $C_{\bullet \bullet}$ be a double complex such that $C_{p,q} = 0$ unless $0 \leq p, q \leq n$.
For any $p, q \in \mathbb{Z}$, we denote the boundary maps $C_{p,q} \lra C_{p-1,q}$ and $C_{p,q} \lra C_{p,q-1}$ by $d'_{p,q}$ and $d''_{p,q}$, respectively.
We assume that $C_{p \bullet}$ and $C_{\bullet q}$ are acyclic for $0 \leq p, q \leq n$.
Let $T_{\bullet}$ be the total complex of $C_{\bullet \bullet}$ and let $d_{\bullet}$ be its boundary map, that is, if $\xi \in C_{p,q} \subseteq T_{r} ~ (p + q = r)$, then
\[
d_{r}(\xi) = (-1)^{p} \cdot d''_{p,q}(\xi) + d'_{p,q}(\xi) \in C_{p,q-1} \oplus C_{p-1,q} \subseteq T_{r-1}.
\]
Then the following assertions hold.
\begin{itemize}
\item[{\rm (1)}]
Suppose that $\xi_{n} \in C_{n,0}$ and $\xi_{n-1} \in C_{n-1,1}$ such that $d'_{n,0}(\xi_{n}) = (-1)^{n} \cdot d''_{n-1,1}(\xi_{n-1})$ are given.
Then there exist elements $\xi_{p} \in C_{p,n-p}$ for $0 \leq \forall p \leq n-2$ such that
\begin{multline*}
\quad \xi_{n} + \xi_{n-1} + \xi_{n-2} + \dots + \xi_{0} \in {\rm Ker} \, d_{n} \\
\subseteq T_{n} = C_{n,0} \oplus C_{n-1,1} \oplus C_{n-2,2} \oplus \dots \oplus C_{0,n}. \quad
\end{multline*}
\item[{\rm (2)}]
Suppose that $\xi_{n} + \xi_{n-1} + \dots + \xi_{1} + \xi_{0} \in {\rm Ker} \, d_{n} \subseteq T_{n} = C_{n,0} \oplus C_{n-1,1} \oplus \dots \oplus C_{1,n-1} \oplus C_{0,n}$
and $\xi_{0} \in {\rm Im} \, d'_{1,n}$. Then
\[
\xi_{n} + \xi_{n-1} + \dots + \xi_{1} + \xi_{0} \in {\rm Im} \, d_{n+1}.
\]
In particular, we have $\xi_{n} \in {\rm Im} \, d''_{n,1}$.
\end{itemize}
\end{lem}

\begin{proof}
(1)\,
It is enough to show that if $1 \leq p \leq n-1$ and two elements $\xi_{p+1} \in C_{p+1,n-p-1}$, $\xi_{p} \in C_{p,n-p}$ such that
\[
d'_{p+1,n-p-1}(\xi_{p+1}) = (-1)^{p+1} \cdot d''_{p,n-p}(\xi_{p})
\]
are given, then we can take $\xi_{p-1} \in C_{p-1,n-p+1}$ so that
\[
d'_{p,n-p}(\xi_{p}) = (-1)^{p} \cdot d''_{p-1,n-p+1}(\xi_{p-1}).
\]
In fact, if the assumption of the claim stated above is satisfied, we have
\begin{align*}
d''_{p-1,n-p}(d'_{p,n-p}(\xi_{p})) &= d'_{p,n-p-1}(d''_{p,n-p}(\xi_{p})) \\
                                   &= d'_{p,n-p-1}((-1)^{p+1} \cdot d'_{p+1,n-p-1}(\xi_{p+1})) \\
                                   &= 0,
\end{align*}
and so
\[
d'_{p,n-p}(\xi_{p}) \in {\rm Ker} \, d''_{p-1,n-p} = {\rm Im} \, d''_{p-1,n-p+1},
\]
which means the existence of the required element $\xi_{p-1}$.

(2)\,
We set $\eta_{0} = 0$. By the assumption, there exists $\eta_{1} \in C_{1,n}$ such that
\[
\xi_{0} = d'_{1,n}(\eta_{1}) = d'_{1,n}(\eta_{1}) + d''_{0,n+1}(\eta_{0}).
\]
Here we assume $0 \leq p \leq n-1$ and two elements $\eta_{p} \in C_{p,n-p+1}$, $\eta_{p+1} \in C_{p+1,n-p}$ such that
\[
\xi_{p} = d'_{p+1,n-p}(\eta_{p+1}) + (-1)^{p} \cdot d''_{p,n-p+1}(\eta_{p})
\]
are fixed.
We would like to find $\eta_{p+2} \in C_{p+2,n-p-1}$ such that
\[
\xi_{p+1} = d'_{p+2,n-p-1}(\eta_{p+2}) + (-1)^{p+1} \cdot d''_{p+1,n-p}(\eta_{p+1}).
\]
In fact, as we have
\begin{align*}
  & \; d'_{p+1,n-p-1}(\xi_{p+1} + (-1)^{p} \cdot d''_{p+1,n-p}(\eta_{p+1})) \\
= & \; d'_{p+1,n-p-1}(\xi_{p+1}) + (-1)^{p} \cdot d'_{p+1,n-p-1}(d''_{p+1,n-p}(\eta_{p+1})) \\
= & \; (-1)^{p+1} \cdot d''_{p,n-p}(\xi_{p}) + (-1)^{p} \cdot d''_{p,n-p}(d'_{p+1,n-p}(\eta_{p+1})) \\
= & \; (-1)^{p+1} \cdot d''_{p,n-p}(\xi_{p} - d'_{p+1,n-p}(\eta_{p+1})) \\
= & \; (-1)^{p+1} \cdot d''_{p,n-p}((-1)^{p} \cdot d''_{p,n-p+1}(\eta_{p})) \\
= & \; 0,
\end{align*}
it follows that
\[
\xi_{p+1} + (-1)^{p} \cdot d''_{p+1,n-p}(\eta_{p+1}) \in {\rm Ker} \, d'_{p+1,n-p-1} = {\rm Im} \, d'_{p+2,n-p-1}.
\]
Thus we see the existence of the required element $\eta_{p+2}$.
\end{proof}

\begin{lem}\label{3}
Suppose that
\[
0 \lra F \stackrel{\varphi}{\lra} G \stackrel{\psi}{\lra} H \stackrel{\rho}{\lra} L
\]
is an exact sequence of $R$-modules. Then the following assertions hold.
\begin{itemize}
\item[{\rm (1)}]
If there exists a homomorphism $\phi : G \lra F$ of $R$-modules such that $\phi \circ \varphi = {\rm id}_{F}$, then
\[
0 \lra \lsa{G} \stackrel{\lsa{\psi}}{\lra} H \stackrel{\rho}{\lra} L
\]
is exact, where $\lsa{G} = {\rm Ker} \, \phi$ and $\lsa{\psi}$ is the restriction of $\psi$ to $\lsa{G}$.
\item[{\rm (2)}]
If $F = \lsp{F} \oplus \lsa{F}$, $G = \lsp{G} \oplus \lsa{G}$, $\varphi(\lsp{F}) = \lsp{G}$ and $\varphi(\lsa{F}) \subseteq \lsa{G}$, then
\[
0 \lra \lsa{F} \stackrel{\lsa{\varphi}}{\lra} \lsa{G} \stackrel{\lsa{\psi}}{\lra} H \stackrel{\rho}{\lra} L
\]
is exact, where $\lsa{\varphi}$ and $\lsa{\psi}$ are the restrictions of $\varphi$ and $\psi$ to $\lsa{F}$ and $\lsa{G}$, respectively.
\end{itemize}
\end{lem}

\begin{proof}
See \cite[2.3]{fin1}.
\end{proof}

\section{$\ast$-transform}

Let $2 \leq n \in \mathbb{Z}$ and let $R$ be an $n$-dimensional Cohen-Macaulay local ring with the maximal ideal $\mathfrak{m}$. Suppose that an acyclic complex
\[
0 \lra F_{n} \stackrel{\varphi_{n}}{\lra} F_{n-1} \lra \dotsb \lra F_{1} \stackrel{\varphi_{1}}{\lra} F_{0}
\]
of finitely generated free $R$-modules such that ${\rm Im} \, \varphi_{n} \subseteq QF_{n-1}$ is given, where $Q = (x_{1}, x_{2}, \dots, x_{n})R$ is a parameter ideal of $R$.
We put $M = {\rm Im} \, \varphi_{1}$, which is an $R$-submodule of $F_{0}$.
In this section, transforming $F_{\bullet}$ suitably, we aim to construct an acyclic complex
\[
0 \lra \lsa{F_{n}} \stackrel{\lsa{\varphi_{n}}}{\lra} \lsa{F_{n-1}} \lra \dotsb \lra \lsa{F_{1}} \stackrel{\lsa{\varphi_{1}}}{\lra} \lsa{F_{0}} = F_{0}
\]
of finitely generated free $R$-modules such that ${\rm Im} \, \lsa{\varphi_{n}} \subseteq \mathfrak{m} \cdot \lsa{F_{n-1}}$ and ${\rm Im} \, \lsa{\varphi_{1}} = M :_{F_{0}} Q$.
Let us call $\lsa{F_{\bullet}}$ the $\ast$-transform of $F_{\bullet}$ with respect to $x_{1}, x_{2}, \dots, x_{n}$.

In this operation, we use the Koszul complex $K_{\bullet} = K_{\bullet}(x_{1}, x_{2}, \dots, x_{n})$.
We denote the boundary map of $K_{\bullet}$ by $\partial_{\bullet}$.
Let $e_{1}, e_{2}, \dots, e_{n}$ be an $R$-free basis of $K_{1}$ such that $\partial_{1}(e_{i}) = x_{i}$ for $1 \leq \forall i \leq n$.
Moreover, we use the following notation:
\begin{itemize}
\item
$N := \{ 1, 2, \dots, n \}$.
\item
$N_{p} := \{ I \subseteq N \, | \, \sharp \, I = p \}$ for $1 \leq \forall p \leq n$ and $N_{0} := \{ \emptyset \}$.
\item
If $1 \leq p \leq n$ and $I = \{ i_{1}, i_{2}, \dots, i_{p} \} \in N_{p}$, where $1 \leq i_{1} < i_{2} < \dots < i_{p} \leq n$, we set
\[
e_{I} = e_{i_{1}} \wedge e_{i_{2}} \wedge \dots \wedge e_{i_{p}} \in K_{p}.
\]
In particular, for $1 \leq \forall i \leq n$, $\check{e}_{i} := e_{N \setminus \{ i \}}$.
Furthermore, $e_{\emptyset}$ denotes the identity element $1_{R}$ of $R = K_{0}$.
\item
If $1 \leq p \leq n$, $I \in N_{p}$ and $i \in I$, we set
\[
s(i, I) = \sharp \, \{ j \in I \, | \, j < i \}.
\]
We define $\sharp \, \emptyset = 0$, so $s(i, I) = 0$ if $i = \min I$.
\end{itemize}
Then, for $0 \leq \forall p \leq n$, $\{ e_{I} \}_{I \in N_{p}}$ is an $R$-free basis of $K_{p}$ and
\[
\partial_{p}(e_{I}) = \sum_{i \in I} (-1)^{s(i, I)} \cdot x_{i} \cdot e_{I \setminus \{ i \}}.
\]
As $x_{1}, x_{2}, \dots, x_{n}$ is an $R$-regular sequence, $K_{\bullet}$ gives an $R$-free resolution of $R/Q$.
Hence, for any $R$-module $N$, we have the following commutative diagram;
\[
\begin{array}{cccccccc}
\! {\rm Hom}_{R}(K_{n-1}, N) \! & \! \stackrel{{\rm Hom}_{R}(\partial_{n}, N)}{\lra} \! & \! {\rm Hom}_{R}(K_{n}, N) \! & \! \lra \! & \! {\rm Ext}_{R}^{n}(R/Q, N) \! & \! \lra \! & \! 0 \! & {\rm (ex)} \\
\big\downarrow\vcenter{\rlap{$\scriptstyle{\cong}$}} & & \big\downarrow\vcenter{\rlap{$\scriptstyle{\cong}$}} & & & & & \\
N^{\oplus n} & \! \! \! \! \! \! \! \! \! \! \! \! \stackrel{(x_{1} \,\, -x_{2} \,\, \dotsb \,\, (-1)^{n-1} \cdot x_{n})}{\lra} \! \! \! \! \! \! \! \! \! \! \! \! & N & \! \lra \! & N/QN & \! \lra \! & \! 0 \! & {\rm (ex)} \rlap{,}
\end{array}
\]
which implies ${\rm Ext}_{R}^{n}(R/Q, N) \cong N/QN$. Using this fact, we show the next result.

\begin{thm}\label{4}
$(M :_{F_{0}} Q)/M \cong F_{n}/QF_{n}$.
\end{thm}

\begin{proof}
We put $L_{0} = F_{0}/M$.
Moreover, for $1 \leq \forall p \leq n-1$, we put $L_{p} = {\rm Im} \, \varphi_{p} \subseteq F_{p-1}$ and consider the exact sequence
\[
0 \lra L_{p} \lra F_{p-1} \stackrel{\varphi_{p-1}}{\lra} L_{p-1} \lra 0,
\]
where $\varphi_{0} : F_{0} \lra L_{0}$ is the canonical surjection. Because
\[
{\rm Ext}_{R}^{p-1}(R/Q, F_{p-1}) = {\rm Ext}_{R}^{p}(R/Q, F_{p-1}) = 0,
\]
we get
\[
{\rm Ext}_{R}^{p}(R/Q, L_{p}) \cong {\rm Ext}_{R}^{p-1}(R/Q, L_{p-1}).
\]
Therefore ${\rm Ext}_{R}^{n-1}(R/Q, L_{n-1}) \cong {\rm Hom}_{R}(R/Q, F_{0}/M) \cong (M :_{F_{0}} Q)/M$.
Next, we look at the exact sequence
\[
0 \lra F_{n} \stackrel{\varphi_{n}}{\lra} F_{n-1} \stackrel{\varphi_{n-1}}{\lra} L_{n-1} \lra 0,
\]
which yields the following commutative diagram;
\[
\begin{array}{cccccccc}
\! 0 \! & \! \lra \! & \! {\rm Ext}_{R}^{n-1}(R/Q, L_{n-1}) \! & \! \lra \! & \! {\rm Ext}_{R}^{n}(R/Q, F_{n}) \! & \! \stackrel{\widetilde{\varphi_{n}}}{\lra} \! & \! {\rm Ext}_{R}^{n}(R/Q, F_{n-1}) \! & {\rm (ex)} \\
 & & & & \big\downarrow\vcenter{\rlap{$\scriptstyle{\cong}$}} & & \big\downarrow\vcenter{\rlap{$\scriptstyle{\cong}$}} \\
 & & & & F_{n}/QF_{n} & \! \stackrel{\overline{\varphi_{n}}}{\lra} \! & F_{n-1}/QF_{n-1} \rlap{,}
\end{array}
\]
where $\widetilde{\varphi_{n}}$ and $\overline{\varphi_{n}}$ denote the maps induced from $\varphi_{n}$.
Let us notice $\overline{\varphi_{n}} = 0$ as ${\rm Im} \, \varphi_{n} \subseteq QF_{n-1}$. Hence
\[
{\rm Ext}_{R}^{n-1}(R/Q, L_{n-1}) \cong F_{n}/QF_{n},
\]
and so the required isomorphism follows.
\end{proof}

Let us fix an $R$-free basis of $F_{n}$, say $\{ v_{\lambda} \}_{\lambda \in \Lambda}$.
We set $\widetilde{\Lambda} = \Lambda \times N$ and take a family $\{ v_{(\lambda, i)} \}_{(\lambda, i) \in \widetilde{\Lambda}}$ of elements in $F_{n-1}$ so that
\[
\varphi_{n}(v_{\lambda}) = \sum_{i \in N} x_{i} \cdot v_{(\lambda, i)}
\]
for $\forall \lambda \in \Lambda$.
This is possible as ${\rm Im} \, \varphi_{n} \subseteq QF_{n-1}$.
The next result is the essential part of the process to get $\lsa{F_{\bullet}}$.

\begin{thm}\label{5}
There exists a chain map $\sigma_{\bullet} : F_{n} \otimes_{R} K_{\bullet} \lra F_{\bullet}$
\[
\begin{array}{ccccccccccc}
\! 0 \! & \! \lra \! & \! F_{n} \otimes_{R} K_{n} \! & \! \stackrel{F_{n} \otimes \partial_{n}}{\lra} \! & \! F_{n} \otimes_{R} K_{n-1} \! & \! \lra \! & \! \dotsb \! & \! \lra \! & \! F_{n} \otimes_{R} K_{1} \! & \! \stackrel{F_{n} \otimes \partial_{1}}{\lra} \! & \! F_{n} \otimes_{R} K_{0} \! \\
 & & \big\downarrow\vcenter{\rlap{$\scriptstyle{\sigma_{n}}$}} & & \big\downarrow\vcenter{\rlap{$\scriptstyle{\sigma_{n-1}}$}} & & & & \big\downarrow\vcenter{\rlap{$\scriptstyle{\sigma_{1}}$}} & & \big\downarrow\vcenter{\rlap{$\scriptstyle{\sigma_{0}}$}} \\
\! 0 \! & \! \lra \! & F_{n} & \! \stackrel{\varphi_{n}}{\lra} \! & F_{n-1} & \! \lra \! & \! \dotsb \! & \! \lra \! & F_{1} & \! \stackrel{\varphi_{1}}{\lra} \! & F_{0}
\end{array}
\]
satisfying the following conditions.
\begin{itemize}
\item[{\rm (1)}]
$\sigma_{0}^{-1}({\rm Im} \, \varphi_{1}) = {\rm Im}(F_{n} \otimes \partial_{1})$.
\item[{\rm (2)}]
${\rm Im} \, \sigma_{0} + {\rm Im} \, \varphi_{1} = M :_{F_{0}} Q$.
\item[{\rm (3)}]
$\sigma_{n-1}(v_{\lambda} \otimes \check{e}_{i}) = (-1)^{n+i-1} \cdot v_{(\lambda, i)}$ for $\forall (\lambda, i) \in \widetilde{\Lambda}$.
\item[{\rm (4)}]
$\sigma_{n}(v_{\lambda} \otimes e_{N}) = (-1)^{n} \cdot v_{\lambda}$ for $\forall \lambda \in \Lambda$.
\end{itemize}
\end{thm}

\begin{proof}
Let us notice that, for $0 \leq \forall p \leq n$, $\{ v_{\lambda} \otimes e_{I} \}_{(\lambda, I) \in \Lambda \times N_{p}}$ is an $R$-free basis of $F_{n} \otimes_{R} K_{p}$,
so $\sigma_{p} : F_{n} \otimes_{R} K_{p} \lra F_{p}$ can be defined by choosing suitable element $w_{(\lambda, I)} \in F_{p}$
that corresponds to $v_{\lambda} \otimes e_{I}$ for $\forall (\lambda, I) \in \Lambda \times N_{p}$.
We set $w_{(\lambda, N)} = (-1)^{n} \cdot v_{\lambda}$ for $\forall \lambda \in \Lambda$
and $w_{(\lambda, N \setminus \{ i \})} = (-1)^{n+i-1} \cdot v_{(\lambda, i)}$ for $\forall (\lambda, i) \in \widetilde{\Lambda}$. Then
\begin{align*}
\varphi_{n}(w_{(\lambda, N)}) & = (-1)^{n} \cdot \varphi_{n}(v_{\lambda}) \\
                              & = (-1)^{n} \cdot \sum_{i \in N} x_{i} \cdot v_{(\lambda, i)} \\
                              & = \sum_{i \in N} (-1)^{s(i, N)} \cdot x_{i} \cdot w_{(\lambda, N \setminus \{ i \})}.
\end{align*}
Moreover, we can take families $\{ w_{(\lambda, I)} \}_{(\lambda, I) \in \Lambda \times N_{p}}$ of elements in $F_{p}$ for $0 \leq \forall p \leq n-2$ so that
\[
\varphi_{p}(w_{(\lambda, I)}) = \sum_{i \in I} (-1)^{s(i, I)} \cdot x_{i} \cdot w_{(\lambda, I \setminus \{ i \})} \tag{$\sharp$}
\]
for $1 \leq \forall p \leq n$ and $\forall (\lambda, I) \in \Lambda \times N_{p}$.
If this is true, an $R$-linear map $\sigma_{p} : F_{n} \otimes_{R} K_{p} \lra F_{p}$ is defined by setting $\sigma_{p}(v_{\lambda} \otimes e_{I}) = w_{(\lambda, I)}$ for $\forall (\lambda, I) \in \Lambda \times N_{p}$ and $\sigma_{\bullet} : F_{n} \otimes_{R} K_{\bullet} \lra F_{\bullet}$ becomes a chain map satisfying (3) and (4).

In order to see the existence of $\{ w_{(\lambda, I)} \}_{(\lambda, I) \in \Lambda \times N_{p}}$, let us consider the double complex $F_{\bullet} \otimes_{R} K_{\bullet}$.
\[
\begin{array}{ccccccc}
 & & \vdots & & \vdots & & \\
 & & \big\downarrow & & \big\downarrow & & \\
\! \dotsb \! & \! \lra \! & \! F_{p} \otimes_{R} K_{q} \! & \! \stackrel{\varphi_{p} \otimes K_{q}}{\lra} \! & \! F_{p-1} \otimes_{R} K_{q} \! & \! \lra \! & \! \dotsb \! \\
 & & \big\downarrow\vcenter{\rlap{$\scriptstyle{F_{p} \otimes \partial_{q}}$}} & & \big\downarrow\vcenter{\rlap{$\scriptstyle{F_{p-1} \otimes \partial_{q}}$}} & & \\
\! \dotsb \! & \! \lra \! & \! F_{p} \otimes_{R} K_{q-1} \! & \! \stackrel{\varphi_{p} \otimes K_{q-1}}{\lra} \! & \! F_{p-1} \otimes_{R} K_{q-1} \! & \! \lra \! & \! \dotsb \! \\
 & & \big\downarrow & & \big\downarrow & & \\
 & & \vdots & & \vdots & &
\end{array}
\]
We can take it as $C_{\bullet \bullet}$ of \ref{2}. Let $T_{\bullet}$ be the total complex and $d_{\bullet}$ be its boundary map. In particular, we have
\[
T_{n} = (F_{n} \otimes_{R} K_{0}) \oplus (F_{n-1} \otimes_{R} K_{1}) \oplus \dots \oplus (F_{1} \otimes_{R} K_{n-1}) \oplus (F_{0} \otimes_{R} K_{n}).
\]
For $\forall I \subseteq N$, we define
\[
t(I) = \begin{cases}
        \displaystyle \sum_{i \in I} (i-1) & \text{if $I \neq \emptyset$}, \\
        \\
        0                                  & \text{if $I = \emptyset$}.
        \end{cases}
\]
For a while, we fix $\lambda \in \Lambda$ and set
\begin{align*}
\xi_{n}(\lambda)   &= (-1)^{\frac{n(n+1)}{2}} \cdot (-1)^{t(N)} \cdot w_{(\lambda, N)} \otimes e_{\emptyset} \in F_{n} \otimes_{R} K_{0}, \\
\xi_{n-1}(\lambda) &= (-1)^{\frac{(n-1)n}{2}} \cdot \sum_{i \in N} (-1)^{t(N \setminus \{ i \})} \cdot w_{(\lambda, N \setminus \{ i \})} \otimes e_{i} \in F_{n-1} \otimes_{R} K_{1}.
\end{align*}
It is easy to see that
\[
\xi_{n}(\lambda) = v_{\lambda} \otimes e_{\emptyset}
\]
since $t(N) = (n-1)n/2$ and $n^{2} + n \equiv 0 \pmod{2}$. Moreover, we have
\[
\xi_{n-1}(\lambda) = (-1)^{n} \cdot \sum_{i \in N} v_{(\lambda, i)} \otimes e_{i}
\]
since $t(N \setminus \{ i \}) = (n-1)n/2 - (i-1)$. Then
\begin{align*}
(\varphi_{n} \otimes K_{0})(\xi_{n}(\lambda)) &= \varphi_{n}(v_{\lambda}) \otimes e_{\emptyset} \\
                                              &= (\sum_{i \in N} x_{i} \cdot v_{(\lambda, i)}) \otimes e_{\emptyset} \\
                                              &= \sum_{i \in N} v_{(\lambda, i)} \otimes x_{i} \\
                                              &= (F_{n-1} \otimes \partial_{1})(\sum_{i \in N} v_{(\lambda, i)} \otimes e_{i}) \\
                                              &= (-1)^{n} \cdot (F_{n-1} \otimes \partial_{1})(\xi_{n-1}(\lambda)).
\end{align*}
Hence, by (1) of \ref{2} there exist elements $\xi_{p}(\lambda) \in F_{p} \otimes K_{n-p}$ for $0 \leq \forall p \leq n-2$ such that
\[
\xi_{n}(\lambda) + \xi_{n-1}(\lambda) + \xi_{n-2}(\lambda) + \dots + \xi_{0}(\lambda) \in {\rm Ker} \, d_{n} \subseteq T_{n},
\]
which means
\[
(\varphi_{p} \otimes K_{n-p})(\xi_{p}(\lambda)) = (-1)^{p} \cdot (F_{p-1} \otimes \partial_{n-p+1})(\xi_{p-1}(\lambda))
\]
for $1 \leq \forall p \leq n$. Let us denote $N \setminus I$ by $I^{\text{c}}$ for $\forall I \subseteq N$.
Because $\{ e_{I^{\text{c}}} \}_{I \in N_{p}}$ is an $R$-free basis of $K_{n-p}$, it is possible to write
\[
\xi_{p}(\lambda) = (-1)^{\frac{p(p+1)}{2}} \cdot \sum_{I \in N_{p}} (-1)^{t(I)} \cdot w_{(\lambda, I)} \otimes e_{I^{\text{c}}}
\]
for $0 \leq \forall p \leq n-2$ (Notice that $\xi_{n}(\lambda)$ and $\xi_{n-1}(\lambda)$ are defined so that they satisfy the same equalities), where $w_{(\lambda, I)} \in F_{p}$. Then we have
\[
(\varphi_{p} \otimes K_{n-p})(\xi_{p}(\lambda)) = (-1)^{\frac{p(p+1)}{2}} \cdot \sum_{I \in N_{p}} (-1)^{t(I)} \cdot \varphi_{p}(w_{(\lambda, I)}) \otimes e_{I^{\text{c}}}.
\]
On the other hand,
\begin{multline*}
(-1)^{p} \cdot (F_{p-1} \otimes \partial_{n-p+1})(\xi_{p-1}(\lambda)) \\
= (-1)^{p} \cdot (-1)^{\frac{(p-1)p}{2}} \cdot \sum_{J \in N_{p-1}} \{ (-1)^{t(J)} \cdot w_{(\lambda, J)} \otimes (\sum_{i \in J^{\text{c}}} (-1)^{s(i, J^{\text{c}})} \cdot x_{i} \cdot e_{J^{\text{c}} \setminus \{ i \}}) \}.
\end{multline*}
Here we notice that if $I \in N_{p}$, $J \in N_{p-1}$ and $i \in N$, then
\[
I^{\text{c}} = J^{\text{c}} \setminus \{ i \} \quad \Longleftrightarrow \quad I = J \cup \{ i \}.
\]
Hence we get
\begin{multline*}
(-1)^{p} \cdot (F_{p-1} \otimes \partial_{n-p+1})(\xi_{p-1}(\lambda)) \\
= (-1)^{\frac{p(p+1)}{2}} \cdot \sum_{I \in N_{p}} \{ (\sum_{i \in I} (-1)^{t(I \setminus \{ i \}) + s(i, I^{\text{c}} \cup \{ i \})} \cdot x_{i} \cdot w_{(\lambda, I \setminus \{ i \})}) \otimes e_{I^{\text{c}}} \}.
\end{multline*}
For $\forall I \in N_{p}$ and $\forall i \in I$, we have
\begin{gather*}
t(I \setminus \{ i \}) = t(I) - (i - 1), \\
s(i, I) + s(i, I^{\text{c}} \cup \{ i \}) = s(i, N) = i - 1,
\end{gather*}
and so
\begin{align*}
t(I \setminus \{ i \}) + s(i, I^{\text{c}} \cup \{ i \}) &= t(I) - s(i, I) \\
                                                         &\equiv t(I) + s(i, I) \pmod{2}.
\end{align*}
Therefore we see that the required equality ($\sharp$) holds for $\forall I \in N_{p}$.

Let us prove (1). We have to show $\sigma_{0}^{-1}(\text{Im} \, \varphi_{1}) \subseteq \text{Im}(F_{n} \otimes \partial_{1})$.
Take $\forall \eta_{n} \in F_{n} \otimes_{R} K_{0}$ such that $\sigma_{0}(\eta_{n}) \in \text{Im} \, \varphi_{1}$.
As $\{ \xi_{n}(\lambda) \}_{\lambda \in \Lambda}$ is an $R$-free basis of $F_{n} \otimes_{R} K_{0}$, we can express
\[
\eta_{n} = \sum_{\lambda \in \Lambda} a_{\lambda} \cdot \xi_{n}(\lambda) = \sum_{\lambda \in \Lambda} a_{\lambda} \cdot (v_{\lambda} \otimes e_{\emptyset}),
\]
where $a_{\lambda} \in R$ for $\forall \lambda \in \Lambda$. Then we have
\[
\sum_{\lambda \in \Lambda} a_{\lambda} \cdot w_{(\lambda, \emptyset)} = \sum_{\lambda \in \Lambda} a_{\lambda} \cdot \sigma_{0}(v_{\lambda} \otimes e_{\emptyset}) = \sigma_{0}(\eta_{n}) \in \text{Im} \, \varphi_{1}.
\]
Now we set
\[
\eta_{p} = \sum_{\lambda \in \Lambda} a_{\lambda} \cdot \xi_{p}(\lambda) \in F_{p} \otimes_{R} K_{n-p}
\]
for $0 \leq \forall p \leq n-1$. Then
\begin{align*}
\eta_{n} + \eta_{n-1} + \dots + \eta_{1} + \eta_{0} &= \sum_{\lambda \in \Lambda} a_{\lambda} \cdot (\xi_{n}(\lambda) + \xi_{n-1}(\lambda) + \dots + \xi_{1}(\lambda) + \xi_{0}(\lambda)) \\
                                                    &\in \text{Ker} \, d_{n} \subseteq T_{n}.
\end{align*}
Because
\begin{align*}
\eta_{0} &= \sum_{\lambda \in \Lambda} a_{\lambda} \cdot \xi_{0}(\lambda) \\
         &= \sum_{\lambda \in \Lambda} a_{\lambda} \cdot (w_{(\lambda, \emptyset)} \otimes e_{N}) \\
         &= (\sum_{\lambda \in \Lambda} a_{\lambda} \cdot w_{(\lambda, \emptyset)}) \otimes e_{N} \\
         &\in \text{Im}(\varphi_{1} \otimes K_{n}),
\end{align*}
we get $\eta_{n} \in \text{Im}(F_{n} \otimes \partial_{1})$ by (2) of \ref{2}.

Finally we prove (2). Let us consider the following commutative diagram
\[
\begin{array}{cccccccc}
\! F_{n} \otimes_{R} K_{1} \! & \! \stackrel{F_{n} \otimes \partial_{1}}{\lra} \! & \! F_{n} \otimes_{R} K_{0} \! & \! \lra \! & \! F_{n}/QF_{n} \! & \! \lra \! & \! 0 \! & \text{(ex)} \\
\big\downarrow\vcenter{\rlap{$\scriptstyle{\sigma_{1}}$}} & & \big\downarrow\vcenter{\rlap{$\scriptstyle{\sigma_{0}}$}} & & \big\downarrow\vcenter{\rlap{$\scriptstyle{\overline{\sigma_{0}}}$}} & & & \\
F_{1} & \! \stackrel{\varphi_{1}}{\lra} \! & F_{0} & \! \lra \! & \! F_{0}/M \! & \! \lra \! & \! 0 \! & \text{(ex)} \rlap{,}
\end{array}
\]
where $\overline{\sigma_{0}}$ is the map induced from $\sigma_{0}$. For $\forall \lambda \in \Lambda$ and $\forall i \in N$, we have
\[
x_{i} \cdot w_{(\lambda, \emptyset)} = \varphi_{1}(w_{(\lambda, \{ i \})}) \in M,
\]
which means $w_{(\lambda, \emptyset)} \in M :_{F_{0}} Q$.
Hence $\text{Im} \, \sigma_{0} \subseteq M :_{F_{0}} Q$, and so $\text{Im} \, \overline{\sigma_{0}} \subseteq (M :_{F_{0}} Q)/M$.
On the other hand, as $\sigma_{0}^{-1}(\text{Im} \, \varphi_{1}) = \text{Im}(F_{n} \otimes \partial_{1})$, we see that $\overline{\sigma_{0}}$ is injective.
Therefore we get $\text{Im} \, \overline{\sigma_{0}} = (M :_{F_{0}} Q)/M$ since $(M :_{F_{0}} Q)/M \cong F_{n}/QF_{n}$ by \ref{4} and $F_{n}/QF_{n}$ has a finite length.
Thus the assertion (2) follows and the proof is complete.
\end{proof}

In the rest, $\sigma_{\bullet} : F_{n} \otimes_{R} K_{\bullet} \lra F_{\bullet}$ is the chain map constructed in \ref{5}.
Then, by \ref{1} the mapping cone $\text{Cone}(\sigma_{\bullet})$ gives an $R$-free resolution of $M :_{F_{0}} Q$, that is,
\begin{multline*}
0
\lra
F_{n} \otimes_{R} K_{n}
\stackrel{\psi_{n+1}}{\lra}
\begin{matrix} F_{n} \otimes_{R} K_{n-1} \\ \oplus \\ F_{n} \end{matrix}
\stackrel{\psi_{n}}{\lra}
\begin{matrix} F_{n} \otimes_{R} K_{n-2} \\ \oplus \\ F_{n-1} \end{matrix}
\stackrel{\lsp{\varphi_{n-1}}}{\lra}
\begin{matrix} F_{n} \otimes_{R} K_{n-3} \\ \oplus \\ F_{n-2} \end{matrix}
\\
\stackrel{\lsa{\varphi_{n-2}}}{\lra}
\begin{matrix} F_{n} \otimes_{R} K_{n-4} \\ \oplus \\ F_{n-3} \end{matrix}
\lra
\dotsb
\lra
\begin{matrix} F_{n} \otimes_{R} K_{1} \\ \oplus \\ F_{2} \end{matrix}
\stackrel{\lsa{\varphi_{2}}}{\lra}
\begin{matrix} F_{n} \otimes_{R} K_{0} \\ \oplus \\ F_{1} \end{matrix}
\stackrel{\lsa{\varphi_{1}}}{\lra}
F_{0}
\end{multline*}
is acyclic and $\text{Im} \, \lsa{\varphi_{1}} = M :_{F_{0}} Q$, where
\begin{gather*}
\psi_{n+1} = \begin{pmatrix} F_{n} \otimes \partial_{n} \\ (-1)^{n} \cdot \sigma_{n} \end{pmatrix}, ~
\psi_{n} = \begin{pmatrix} F_{n} \otimes \partial_{n-1} & 0 \\ (-1)^{n-1} \cdot \sigma_{n-1} & \varphi_{n} \end{pmatrix}, ~
\lsp{\varphi_{n-1}} = \begin{pmatrix} F_{n} \otimes \partial_{n-2} & 0 \\ (-1)^{n-2} \cdot \sigma_{n-2} & \varphi_{n-1} \end{pmatrix},
\\
\lsa{\varphi_{p}} = \begin{pmatrix} F_{n} \otimes \partial_{p-1} & 0 \\ (-1)^{p-1} \cdot \sigma_{p-1} & \varphi_{p} \end{pmatrix} ~ \text{for} ~ 2 \leq \forall p \leq n-2 ~ \text{and} ~
\lsa{\varphi_{1}} = \begin{pmatrix} \sigma_{0} & \varphi_{1} \end{pmatrix}.
\end{gather*}
Because $\sigma_{n} : F_{n} \otimes_{R} K_{n} \lra F_{n}$ is an isomorphism, we can define
\[
\phi = \begin{pmatrix} 0 & (-1)^{n} \cdot \sigma_{n}^{-1} \end{pmatrix} : \begin{matrix} F_{n} \otimes_{R} K_{n-1} \\ \oplus \\ F_{n} \end{matrix} \lra F_{n} \otimes_{R} K_{n}.
\]
Then $\phi \circ \psi_{n+1} = \text{id}_{F_{n} \otimes_{R} K_{n}}$ and $\text{Ker} \, \phi = F_{n} \otimes_{R} K_{n-1}$.
Hence, by (1) of \ref{3}, we get the acyclic complex
\[
0
\lra
\lsp{F_{n}}
\stackrel{\lsp{\varphi_{n}}}{\lra}
\lsp{F_{n-1}}
\stackrel{\lsp{\varphi_{n-1}}}{\lra}
\lsa{F_{n-2}}
\stackrel{\lsa{\varphi_{n-2}}}{\lra}
\lsa{F_{n-3}}
\lra
\dotsb
\lra
\lsa{F_{2}}
\stackrel{\lsa{\varphi_{2}}}{\lra}
\lsa{F_{1}}
\stackrel{\lsa{\varphi_{1}}}{\lra}
\lsa{F_{0}} = F_{0},
\]
where
\begin{gather*}
\lsp{F_{n}} = F_{n} \otimes_{R} K_{n-1}, ~
\lsp{F_{n-1}} = \begin{matrix} F_{n} \otimes_{R} K_{n-2} \\ \oplus \\ F_{n-1} \end{matrix}, ~
\lsa{F_{p}} = \begin{matrix} F_{n} \otimes_{R} K_{p-1} \\ \oplus \\ F_{p} \end{matrix} ~ \text{for} ~ 1 \leq \forall p \leq n-2
\\
\text{and} ~ \lsp{\varphi_{n}} = \begin{pmatrix} F_{n} \otimes \partial_{n-1} \\ (-1)^{n-1} \cdot \sigma_{n-1} \end{pmatrix}.
\end{gather*}

Although $\text{Im} \, \lsp{\varphi_{n}}$ may not be contained in $\mathfrak{m} \cdot \lsp{F_{n-1}}$,
removing non-minimal components from $\lsp{F_{n}}$ and $\lsp{F_{n-1}}$,
we get free $R$-modules $\lsa{F_{n}}$ and $\lsa{F_{n-1}}$ such that
\[
0
\lra
\lsa{F_{n}}
\stackrel{\lsa{\varphi_{n}}}{\lra}
\lsa{F_{n-1}}
\stackrel{\lsa{\varphi_{n-1}}}{\lra}
\lsa{F_{n-2}}
\lra
\dotsb
\lra
\lsa{F_{1}}
\stackrel{\lsa{\varphi_{1}}}{\lra}
\lsa{F_{0}} = F_{0}
\]
is acyclic and $\text{Im} \, \lsa{\varphi_{n}} \subseteq \mathfrak{m} \cdot \lsa{F_{n-1}}$,
where $\lsa{\varphi_{n}}$ and $\lsa{\varphi_{n-1}}$ are the restrictions of $\lsp{\varphi_{n}}$ and $\lsp{\varphi_{n-1}}$, respectively.
In the rest of this section, we describe a concrete procedure to get $\lsa{F_{n}}$ and $\lsa{F_{n-1}}$. For that purpose, we use the following notation.
As described in Introduction, for any $\xi \in F_{n} \otimes_{R} K_{n-2}$ and $\eta \in F_{n-1}$,
\[
\brac{\xi} := \begin{pmatrix} \xi \\ 0 \end{pmatrix} \in \lsp{F_{n-1}} ~ \text{and} ~ \ang{\eta} := \begin{pmatrix} 0 \\ \eta \end{pmatrix} \in \lsp{F_{n-1}}.
\]
In particular, for any $(\lambda, I) \in \Lambda \times N_{n-2}$, we denote $\brac{v_{\lambda} \otimes e_{I}}$ by $\brac{\lambda, I}$.
Moreover, for a subset $U$ of $F_{n-1}$, $\ang{U} := \{ \ang{u} \}_{u \in U}$.

Now, let us choose a subset $\lsp{\Lambda}$ of $\widetilde{\Lambda}$ and a subset $U$ of $F_{n-1}$ so that
\[
\{ v_{(\lambda, i)} \}_{(\lambda, i) \in \lsp{\Lambda}} \cup U
\]
is an $R$-free basis of $F_{n-1}$. We would like to choose $\lsp{\Lambda}$ as big as possible. The following almost obvious fact is useful to find $\lsp{\Lambda}$ and $U$.

\begin{lem}\label{6}
Let $V$ be an $R$-free basis of $F_{n-1}$. If a subset $\lsp{\Lambda}$ of $\widetilde{\Lambda}$ and a subset $U$ of $V$ satisfy
\begin{itemize}
\item[{\rm (i)}]
$\sharp \, \lsp{\Lambda} + \sharp \, U \leq \sharp \, V$, and
\item[{\rm (ii)}]
$V \subseteq R \cdot \{ v_{(\lambda, i)} \}_{(\lambda, i) \in \lsp{\Lambda}} + R \cdot U + \mathfrak{m} F_{n-1}$,
\end{itemize}
then $\{ v_{(\lambda, i)} \}_{(\lambda, i) \in \lsp{\Lambda}} \cup U$ is an $R$-free basis of $F_{n-1}$.
\end{lem}

Let us notice that
\[
\{ \brac{\lambda, I} \}_{(\lambda, I) \in \Lambda \times N_{n-2}} \cup \{ \ang{v_{(\lambda, i)}} \}_{(\lambda, i) \in \lsp{\Lambda}} \cup \ang{U}
\]
is an $R$-free basis of $\lsp{F_{n-1}}$. We define $\lsa{F_{n-1}}$ to be the direct summand of $\lsp{F_{n-1}}$ generated by
\[
\{ \brac{\lambda, I} \}_{(\lambda, I) \in \Lambda \times N_{n-2}} \cup \ang{U}.
\]
Let $\lsa{\varphi_{n-1}}$ be the restriction of $\lsp{\varphi_{n-1}}$ to $\lsa{F_{n-1}}$.

\begin{thm}\label{7}
If we can take $\widetilde{\Lambda}$ itself as $\lsp{\Lambda}$, then
\[
0
\lra
\lsa{F_{n-1}}
\stackrel{\lsa{\varphi_{n-1}}}{\lra}
\lsa{F_{n-2}}
\lra
\dotsb
\lra
\lsa{F_{1}}
\stackrel{\lsa{\varphi_{1}}}{\lra}
\lsa{F_{0}} = F_{0}
\]
is acyclic. Hence we have ${\rm depth}_{R} \, F_{0}/(M :_{F_{0}} Q) > 0$.
\end{thm}

\begin{proof}
If $\lsp{\Lambda} = \widetilde{\Lambda}$, there exists a homomorphism $\phi : \lsp{F_{n-1}} \lra \lsp{F_{n}}$ such that
\begin{gather*}
\phi(\brac{\lambda, I}) = 0 \quad \text{for any} ~ (\lambda, I) \in \Lambda \times N_{n-2}, \\
\phi(\ang{v_{(\lambda, i)}}) = (-1)^{i} \cdot v_{\lambda} \otimes \check{e}_{i} \quad \text{for any} ~ (\lambda, i) \in \widetilde{\Lambda}, \\
\phi(\ang{u}) = 0 \quad \text{for any} ~ u \in U.
\end{gather*}
Then $\phi \circ \lsp{\varphi_{n}} = \text{id}_{\lsp{F_{n}}}$ and $\text{Ker} \, \phi = \lsa{F_{n-1}}$. Hence, by (1) of \ref{3} we get the required assertion.
\end{proof}

In the rest of this section, we assume $\lsp{\Lambda} \subsetneq \widetilde{\Lambda}$ and put $\lsa{\Lambda} = \widetilde{\Lambda} \setminus \lsp{\Lambda}$.
Then, for any $(\mu, j) \in \lsa{\Lambda}$, it is possible to write
\[
v_{(\mu, j)} = \sum_{(\lambda, i) \in \lsp{\Lambda}} a_{(\lambda, i)}^{(\mu, j)} \cdot v_{(\lambda, i)} + \sum_{u \in U} b_{u}^{(\mu, j)} \cdot u,
\]
where $a_{(\lambda, i)}^{(\mu, j)}, b_{u}^{(\mu, j)} \in R$. Here, if $\lsp{\Lambda}$ is big enough, we can choose every $b_{u}^{(\mu, j)}$ from $\mathfrak{m}$.
In fact, if $b_{u}^{(\mu, j)} \notin \mathfrak{m}$ for some $u \in U$, then we can replace $\lsp{\Lambda}$ and $U$ by $\lsp{\Lambda} \cup \{ (\mu, j) \}$ and $U \setminus \{ u \}$, respectively.
Furthermore, because of a practical reason, let us allow that some terms of $v_{(\lambda, i)}$ for $(\lambda, i) \in \lsa{\Lambda}$ with non-unit coefficients appear in the right hand side, that is,
for any $(\mu, j) \in \lsa{\Lambda}$, we write
\[
v_{(\mu, j)} = \sum_{(\lambda, i) \in \widetilde{\Lambda}} a_{(\lambda, i)}^{(\mu, j)} \cdot v_{(\lambda, i)} + \sum_{u \in U} b_{u}^{(\mu, j)} \cdot u,
\]
where
\[
a_{(\lambda, i)}^{(\mu, j)} \in \begin{cases}
                                R            & \text{if $(\lambda, i) \in \lsp{\Lambda}$}, \\
                                \\
                                \mathfrak{m} & \text{if $(\lambda, i) \in \lsa{\Lambda}$}
                                \end{cases}
~ \text{and} ~ b_{u}^{(\mu, j)} \in \mathfrak{m}.
\]
Using this expression, for any $(\mu, j) \in \lsa{\Lambda}$, the following element in $\lsp{F_{n}}$ can be defined.
\[
\lsa{v_{(\mu, j)}} := (-1)^{j} \cdot v_{\mu} \otimes \check{e}_{j} + \sum_{(\lambda, i) \in \widetilde{\Lambda}} (-1)^{i-1} \cdot a_{(\lambda, i)}^{(\mu, j)} \cdot v_{\lambda} \otimes \check{e}_{i}.
\]

\begin{lem}\label{8}
For any $(\mu, j) \in \lsa{\Lambda}$, we have
\[
\lsp{\varphi_{n}}(\lsa{v_{(\mu, j)}})
= (-1)^{j} \cdot \brac{v_{\mu} \otimes \partial_{n-1}(\check{e}_{j})}
+ \sum_{(\lambda, i) \in \widetilde{\Lambda}} (-1)^{i-1} \cdot a_{(\lambda, i)}^{(\mu, j)} \cdot \brac{v_{\lambda} \otimes \partial_{n-1}(\check{e}_{i})}
+ \sum_{u \in U} b_{u}^{(\mu, j)} \cdot \ang{u}.
\]
As a consequence, we have $\lsp{\varphi_{n}}(\lsa{v_{(\mu, j)}}) \in \mathfrak{m} \cdot \lsa{F_{n-1}}$ for any $(\mu, j) \in \lsa{\Lambda}$.
\end{lem}

\begin{proof}
By the definition of $\lsp{\varphi_{n}}$, for any $(\mu, j) \in \lsa{\Lambda}$, we have
\[
\lsp{\varphi_{n}}(\lsa{v_{(\mu, j)}}) = \brac{(F_{n} \otimes \partial_{n-1})(\lsa{v_{(\mu, j)}})} + \ang{(-1)^{n-1} \cdot \sigma_{n-1}(\lsa{v_{(\mu, j)}})}.
\]
Because
\[
(F_{n} \otimes \partial_{n-1})(\lsa{v_{(\mu, j)}})
= (-1)^{j} \cdot v_{\mu} \otimes \partial_{n-1}(\check{e}_{j}) + \sum_{(\lambda, i) \in \widetilde{\Lambda}} (-1)^{i-1} \cdot a_{(\lambda, i)}^{(\mu, j)} \cdot v_{\lambda} \otimes \partial_{n-1}(\check{e}_{i})
\]
and
\begin{align*}
\sigma_{n-1}(\lsa{v_{(\mu, j)}})
&= (-1)^{j} \cdot \sigma_{n-1}(v_{\mu} \otimes \check{e}_{j}) + \sum_{(\lambda, i) \in \widetilde{\Lambda}} (-1)^{i-1} \cdot a_{(\lambda, i)}^{(\mu, j)} \cdot \sigma_{n-1}(v_{\lambda} \otimes \check{e}_{i}) \\
&= (-1)^{n-1} \cdot v_{(\mu, j)} + (-1)^{n} \cdot \sum_{(\lambda, i) \in \widetilde{\Lambda}} a_{(\lambda, i)}^{(\mu, j)} \cdot v_{(\lambda, i)} \\
&= (-1)^{n-1} \cdot (v_{(\mu, j)} - \sum_{(\lambda, i) \in \widetilde{\Lambda}} a_{(\lambda, i)}^{(\mu, j)} \cdot v_{(\lambda, i)}) \\
&= (-1)^{n-1} \cdot \sum_{u \in U} b_{u}^{(\mu, j)} \cdot u,
\end{align*}
we get the required equality.
\end{proof}

Let $\lsa{F_{n}}$ be the $R$-submodule of $\lsp{F_{n}}$ generated by $\{ \lsa{v_{(\mu, j)}}\}_{(\mu, j) \in \lsa{\Lambda}}$ and let $\lsa{\varphi_{n}}$ be the restriction of $\lsp{\varphi_{n}}$ to $\lsa{F_{n}}$.
By \ref{8} we have $\text{Im} \, \lsa{\varphi_{n}} \subseteq \lsa{F_{n-1}}$. Thus we get a complex
\[
0
\lra
\lsa{F_{n}}
\stackrel{\lsa{\varphi_{n}}}{\lra}
\lsa{F_{n-1}}
\lra
\dotsb
\lra
\lsa{F_{1}}
\stackrel{\lsa{\varphi_{1}}}{\lra}
\lsa{F_{0}} = F_{0}.
\]
This is the complex we desire. In fact, the following result holds.

\begin{thm}\label{9}
$(\lsa{F_{\bullet}}, \lsa{\varphi_{\bullet}})$ is an acyclic complex of finitely generated free $R$-modules with the following properties.
\begin{itemize}
\item[{\rm (1)}]
${\rm Im} \, \lsa{\varphi_{1}} = M :_{F_{0}} Q$ and ${\rm Im} \, \lsa{\varphi_{n}} \subseteq \mathfrak{m} \cdot \lsa{F_{n-1}}$.
\item[{\rm (2)}]
$\{ \lsa{v_{(\mu, j)}}\}_{(\mu, j) \in \lsa{\Lambda}}$ is an $R$-free basis of $\lsa{F_{n}}$.
\item[{\rm (3)}]
$\{ \brac{\lambda, I} \}_{(\lambda, I) \in \Lambda \times N_{n-2}} \cup \ang{U}$ is an $R$-free basis of $\lsa{F_{n-1}}$.
\end{itemize}
\end{thm}

\begin{proof}
First, let us notice that $\{ v_{\lambda} \otimes \check{e}_{i} \}_{(\lambda, i) \in \widetilde{\Lambda}}$ is an $R$-free basis of $\lsp{F_{n}}$ and
\[
v_{\mu} \otimes \check{e}_{j} \in R \cdot \lsa{v_{(\mu, j)}} + R \cdot \{ v_{\lambda} \otimes \check{e}_{i} \}_{(\lambda, i) \in \lsp{\Lambda}} + \mathfrak{m} \cdot \lsp{F_{n}}
\]
for any $(\mu, j) \in \lsa{\Lambda}$. Hence, by Nakayama's lemma it follows that $\lsp{F_{n}}$ is generated by
\[
\{ v_{\lambda} \otimes \check{e}_{i} \}_{(\lambda, i) \in \lsp{\Lambda}} \cup \{ \lsa{v_{(\mu, j)}} \}_{(\mu, j) \in \lsa{\Lambda}},
\]
which must be an $R$-free basis since $\text{rank}_{R} \, \lsp{F_{n}} = \sharp \, \widetilde{\Lambda} = \sharp \, \lsp{\Lambda} + \sharp \, \lsa{\Lambda}$.
Let $\lspp{F_{n}}$ be the $R$-submodule of $\lsp{F_{n}}$ generated by $\{ v_{\lambda} \otimes \check{e}_{i} \}_{(\lambda, i) \in \lsp{\Lambda}}$. Then $\lsp{F_{n}} = \lspp{F_{n}} \oplus \lsa{F_{n}}$.

Next, let us recall that
\[
\{ \brac{\lambda, I} \}_{(\lambda, I) \in \Lambda \times N_{n-2}} \cup \{ \ang{v_{(\lambda, i)}} \}_{(\lambda, i) \in \lsp{\Lambda}} \cup \ang{U}
\]
is an $R$-free basis of $\lsp{F_{n-1}}$. Because
\[
\lsp{\varphi_{n}}(v_{\lambda} \otimes \check{e}_{i}) = \brac{v_{\lambda} \otimes \partial_{n-1}(\check{e}_{i})} + (-1)^{i} \cdot \ang{v_{(\lambda, i)}},
\]
we see that
\[
\{ \brac{\lambda, I} \}_{(\lambda, I) \in \Lambda \times N_{n-2}} \cup \{ \lsp{\varphi_{n}}(v_{\lambda} \otimes \check{e}_{i}) \}_{(\lambda, i) \in \lsp{\Lambda}} \cup \ang{U}
\]
is also an $R$-free basis.
Let $\lspp{F_{n-1}} = R \cdot \{ \lsp{\varphi_{n}}(v_{\lambda} \otimes \check{e}_{i}) \}_{(\lambda, i) \in \lsp{\Lambda}}$. Then $\lsp{F_{n-1}} = \lspp{F_{n-1}} \oplus \lsa{F_{n-1}}$.

It is obvious that $\lsp{\varphi_{n}}(\lspp{F_{n}}) = \lspp{F_{n-1}}$. Moreover, by \ref{8} we get $\lsp{\varphi_{n}}(\lsa{F_{n}}) \subseteq \lsa{F_{n-1}}$.
Therefore, by (2) of \ref{3}, it follows that $\lsa{F_{\bullet}}$ is acyclic. We have already seen (3) and the first assertion of (1). The second assertion of (1) follows from \ref{8}.
Moreover, the assertion (2) is now obvious.
\end{proof}

\end{document}